\documentclass[12pt]{amsart}
\usepackage{hyperref,amssymb}

\theoremstyle{plain}
\newtheorem{theorem}{Theorem}[section]
\newtheorem{proposition}[theorem]{Proposition}
\newtheorem{corollary}[theorem]{Corollary}
\newtheorem{lemma}[theorem]{Lemma}

\theoremstyle{definition}
\newtheorem{remark}[theorem]{Remark}

\newtheorem{example}[theorem]{Example}

\newcommand{\abs}[1]{\lvert#1\rvert}

\newcommand{\bigabs}[1]{\bigl\lvert#1\bigr\rvert}

\newcommand{\term}[1]{{\textit{\textbf{#1}}}}   
\newcommand{\marg}[1]{\marginpar{\tiny #1}}     
\renewcommand{\mid}{\::\:}

\newcommand{\goeso}{\xrightarrow{\mathrm{o}}}	
\newcommand{\goesuo}{\xrightarrow{\mathrm{uo}}}	

\DeclareSymbolFont{bbold}{U}{bbold}{m}{n}
\DeclareSymbolFontAlphabet{\mathbbold}{bbold}
\def\one{\mathbbold{1}}

\DeclareMathOperator{\Int}{Int}
\DeclareMathOperator{\supp}{supp}

\renewcommand{\le}{\leqslant}
\renewcommand{\ge}{\geqslant}

\begin{document}

\title{Order and uo-convergence in spaces of continuous functions}

\author{Eugene Bilokopytov}
\email{bilokopy@ualberta.ca}

\author{Vladimir G. Troitsky}
\email{troitsky@ualberta.ca}

\address{Department of Mathematical and Statistical Sciences,
  University of Alberta, Edmonton, AB, T6G\,2G1, Canada.}

\thanks{The second author was supported by an NSERC grant.}
\keywords{Vector lattice, Banach lattice, order convergence, continuous functions
}
\subjclass[2010]{Primary: 46B42. Secondary: 46A40}

\date{\today}

\begin{abstract}
  We present several characterizations of uo-convergent nets or
  sequences in spaces of continuous functions $C(\Omega)$,
  $C_b(\Omega)$, $C_0(\Omega)$, and $C^\infty(\Omega)$, extending
  results of~\cite{vanderWalt:18}. In particular, it is shown that a
  sequence uo-converges iff it converges pointwise on a co-meagre
  set. We also characterize order bounded sets in spaces of continuous
  functions. This leads to characterizations of order convergence.
\end{abstract}

\maketitle

\section{Introduction}

Being a vector lattice, the space $C(\Omega)$ of continuous functions
on a completely regular Hausdorff topological space is endowed with
order and uo-convergences arising from its order structure. While it
has been known for a long time that for sequences in $L_p$-spaces (and
in other ``nice'' spaces of measurable functions), uo-convergence
agrees with convergence almost everywhere, there have been no similar
characterizations for spaces of continuous functions. There have been
several partial results suggesting that order and uo-convergence in
$C(\Omega)$ somewhat corresponds to pointwise convergence on a ``large
set''. In particular, it was proved in~\cite{vanderWalt:18} using
techniques of semi-continuous functions that for a modified definition
of uo-convergence, a sequence in $C(\Omega)$ uo-converges iff it
converges pointwise outside of a meagre set, under the assumption that
$\Omega$ is a Baire space and $C(\Omega)$ is
almost $\sigma$-order complete. In this paper, we prove this (and
other similar statements) without the almost $\sigma$-order completeness
assumption and for the ``full'' definition of order convergence using
direct topological arguments. Our main tool is Theorem~\ref{CK-o-conv},
which characterizes uo-convergence of nets in $C(\Omega)$ in ``local''
topological terms.

Recall that a net converges in order iff it uo-converges and has an
order bounded tail. In view of this, our characterizations of
uo-convergent nets immediately yield similar
characterizations of order convergent nets.

We then extend these results to ``more specialized'' spaces of
continuous functions: $C_b(\Omega)$, $C_0(\Omega)$, and
$C^\infty(\Omega)$. We also characterize uo-Cauchy nets in these
spaces.

As an application, we deduce a characterization of uo-convergence in
vector lattices with PPP. In the last section, we extend some of the
results to the setting of Boolean algebras.

\section{Preliminaries}

Throughout this note, $X$ will stand for an Archimedean vector
lattice.  We refer the reader to~\cite{Aliprantis:06} for background
on vector lattices.  Recall that a net $(x_\alpha)$ in a vector
lattice \term{converges in order} to $x$ if there exists a net $(u_\gamma)$,
possibly with a different index set, such that $u_\gamma\downarrow 0$,
and for every $\gamma$ there exists $\alpha_0$ such that
$\abs{x_\alpha-x}\le u_\gamma$ for all $\alpha\ge\alpha_0$; we write
$x_\alpha\goeso 0$. The following lemma is standard and
straightforward.

\begin{lemma}\label{oconv-dom-set}
  For a net $(x_\alpha)$ in a vector lattice~$X$, $x_\alpha\goeso 0$
  iff there exists a set $G\subseteq X_+$ such that $\inf G=0$ and
  every element of $G$ dominates a tail of $(x_\alpha)$, i.e., for
  every $g\in G$ there exists $\alpha_0$ such that $\abs{x_\alpha}\le
  g$ for all $\alpha\ge\alpha_0$.
\end{lemma}

We say that $(x_\alpha)$ \term{unbounded order converges} to $x$ or,
short, \term{uo-converges}, and write $x_\alpha\goesuo x$ if
$\abs{x_\alpha-x}\wedge u\goeso 0$ for every $u\ge 0$. We refer the
reader to~\cite{Gao:17} and~\cite{Papangelou:64} for background
information on uo-convergence.  The two convergences agree for order
bounded nets. If $w\ge 0$ is a weak unit then $x_\alpha\goesuo x$ iff
$\abs{x_\alpha-x}\wedge w\goeso 0$. A sublattice $Y$ of $X$ is
\term{regular} if the inclusion map is order continuous, i.e.,
preserves order convergence of nets. Every order dense sublattice is
regular. Given a net $(x_\alpha)$ in a regular sublattice $Y$ of~$X$,
$x_\alpha\goesuo 0$ in $X$ iff $x_\alpha\goesuo 0$ in~$Y$.

Recall that a net $(x_\alpha)_{\alpha\in\Lambda}$
is \term{order Cauchy} if the double net $(x_\alpha-x_\beta)$ indexed by
$\Lambda^2$ is order null. \term{Uo-Cauchy} nets are defined in a similar
fashion.


While all results of the paper are valid for compact Hausdorff spaces,
we will state and prove them in a more general setting.  Throughout
the paper, $\Omega$ stands for a completely regular Hausdorff
topological space (also known as a Tychonov space), which is exactly
the class of Hausdorff spaces where the conclusion of Uryson's
lemma holds. Recall that every locally compact Hausdorff space or every
normal space is completely regular.

We write $C(\Omega)$ for the space of all real-valued continuous
functions on $\Omega$, and $\one$ for the constant one function.  A
subset $A$ of a Hausdorff topological space is \term{nowhere dense} if
$\Int\overline{A}=\varnothing$. It is easy to see that the boundary of
an open set is nowhere dense.  We say that $A$ \term{meagre} or
\term{of first category} if it is a union of a sequence of nowhere
dense sets. The class of meagre sets is closed under taking subsets
and countable unions. A set is \term{co-meagre} if its complement is
meagre. It is easy to see that $A$ is co-meagre iff $A$ contains an
intersection of a sequence of open dense sets. $\Omega$ is a
\term{Baire space} if it satisfies the conclusion of Baire Category
Theorem, i.e., if every intersection of countably many dense open sets
is dense or, equivalently, if every co-magre set is dense. Every
locally compact Hausdorff space and every complete metrizable space is
Baire; see, e.g., \cite[IX.5, Theorem 1]{Bourbaki:98}.

\section{Convergence in $C(\Omega)$}

The following lemma is essentially Proposition~4.13
in~\cite{Bilokopytov}; we provide the proof for completeness.

\begin{lemma}\label{CK-set-inf}
  Suppose that $\Omega$ is a completely regular Hausdorff topological
  space and $G\subseteq C(\Omega)_+$. TFAE:
  \begin{enumerate}
    \item\label{CK-set-inf-inf} $\inf G=0$;
    \item\label{CK-set-inf-t} for every non-empty open set $U$ and every
  $\varepsilon>0$ there exists $t\in U$ and $g\in G$ with
  $g(t)<\varepsilon$;
    \item\label{CK-set-inf-V} for every non-empty open set $U$ and every
  $\varepsilon>0$ there exists a non-empty open set $V\subseteq U$ and
  $g\in G$ such that $g(t)<\varepsilon$ for all $t\in V$.
  \end{enumerate}
\end{lemma}

\begin{proof}
  \eqref{CK-set-inf-inf}$\Rightarrow$\eqref{CK-set-inf-t}
  Suppose that $\inf G=0$, yet \eqref{CK-set-inf-t} fails,
  that is, there exists an open non-empty $U$ and $\varepsilon>0$ such
  that every $g\in G$ is greater than or equal to $\varepsilon$ on
  $U$. Since $\Omega$ is completely regular, we find a non-zero
  $f\in C(\Omega)_+$ such that $f\le\varepsilon\one$ and $f$ vanishes
  outside of~$U$. Then $f\le G$, which contradicts $\inf G=0$.

  To show that
  \eqref{CK-set-inf-t}$\Rightarrow$\eqref{CK-set-inf-inf}, suppose
  that $\inf G\ne 0$. Then there exists $f\in C(\Omega)$ with
  $0<f\le G$. We can then find an open non-empty set $U$ and
  $\varepsilon>0$ such that $f$ is greater that $\varepsilon$
  on~$U$. It follows that every $g\in G$ greater that $\varepsilon$
  on~$U$, which contradicts \eqref{CK-set-inf-t}.
  
  It is easy to see that
  \eqref{CK-set-inf-t}$\Leftrightarrow$\eqref{CK-set-inf-V}.
\end{proof}

Next, we are going to characterize  uo-convergence in $C(\Omega)$.
Since $f_\alpha\goesuo f$ iff $\abs{f_\alpha-f}\goesuo 0$, it suffices
to characterize order convergence of positive nets to zero.

\begin{theorem}\label{CK-o-conv}
  Let $\Omega$ be completely regular and $(f_\alpha)$ a net in $C(\Omega)_+$.
  Then $f_\alpha\goesuo 0$ iff for every non-empty open set $U$ and every
  $\varepsilon>0$ there exists an open non-empty $V\subseteq U$ and an
  index $\alpha_0$ such that $f_\alpha$ is less than $\varepsilon$ on
  $V$ whenever $\alpha\ge\alpha_0$.
\end{theorem}

\begin{proof}
  Suppose $f_\alpha\goesuo 0$. Then $f_\alpha\wedge\one\goeso 0$. By
  Lemma~\ref{oconv-dom-set}, there exists a set $G\subseteq C(\Omega)_+$
  such that $\inf G=0$ and every member of $G$ dominates a tail of
  $(f_\alpha\wedge\one)$. Fix a non-empty open $U$ and
  $\varepsilon\in(0,1)$. Let $V$ and $g$ be as in
  Lemma~\ref{CK-set-inf}\eqref{CK-set-inf-V}. Since $g$ dominates a tail of
  $(f_\alpha\wedge\one)$, there exists $\alpha_0$ such that
  $f_\alpha\wedge\one \le g$ for all $\alpha\ge\alpha_0$. In particular,
  \begin{math}
    f_\alpha(s)\wedge 1\le g(s)<\varepsilon,
  \end{math}
  hence $f_\alpha(s)<\varepsilon$
  for all $s\in V$. This proves the forward implication.

  To prove the converse, suppose that the condition in the theorem is
  satisfied. Since $\one$ is a weak unit, it suffices to prove that
  $f_\alpha\wedge\one\goeso 0$. We will use Lemma~\ref{oconv-dom-set}
  to prove this. Fix an open non-empty set $U$ and
  $\varepsilon>0$. Let $V$ and $\alpha_0$ be as in the assumption. Fix
  any $t\in V$. Since $\Omega$ is completely regular, we can find
  $h\in C(\Omega)_+$ such that $h(t)=0$ and $h$ equals 1 outside
  of~$V$. Put $g=h\vee\varepsilon\one$. Then $g(t)=\varepsilon$. We
  claim that $f_\alpha\wedge\one\le g$ for every
  $\alpha\ge\alpha_0$. Indeed, if $s\in V$ then
  $f_\alpha(s)<\varepsilon\le g(s)$, and if $s\notin V$ then
  $(f_\alpha\wedge\one)(s)\le 1=h(s)\le g(s)$.

  Repeat this process for every pair $(U,\varepsilon)$, where $U$ is
  open and non-empty and $\varepsilon>0$; let $G$ be the set of the
  resulting functions~$g$. Each such $g$ dominates a tail of
  $(f_\alpha\wedge\one)$. Lemma~\ref{CK-set-inf} yields $\inf G=0$; this
  completes the proof.
\end{proof}

\begin{corollary}\label{CK-uo-Cauchy}
  Let $\Omega$ be completely regular and $(f_\alpha)$ a net in
  $C(\Omega)$.  Then $(f_\alpha)$ is uo-Cauchy iff for every non-empty
  open set $U$ and every $\varepsilon>0$ there exists an open
  non-empty $V\subseteq U$ and an index $\alpha_0$ such that
  $\bigabs{f_\alpha(t)-f_\beta(t)}<\varepsilon$ for all $t\in V$ and
  $\alpha,\beta\ge\alpha_0$.
\end{corollary}

Since the space $C_b(\Omega)$ of bounded functions in $C(\Omega)$ is
an order dense and, therefore, regular sublattice of $C(\Omega)$, the
theorem also characterizes uo-convergence in $C_b(\Omega)$. If
$\Omega$ is locally compact, the same is true for $C_0(\Omega)$, the
space of all functions $f$ in $C(\Omega)$ such that for every
$\varepsilon>0$ there exists a compact set $K$ such that
$\bigabs{f(t)}<\varepsilon$ whenever $t\notin K$. It also
follows that a net in $C(\Omega)$ converges in order iff it satisfies
the condition in the theorem and has an order bounded tail; same in
$C_b(\Omega)$ or $C_0(\Omega)$.

The preceding paragraph underlines the importance of understanding
order boundedness in these spaces. Can one characterize order bounded
sets in topological terms? 

It is clear that a set in $C_b(\Omega)$ is order bounded iff it is
bounded in the supremum norm. Next, we characterize order bounded sets
in $C(\Omega)$ under the assumption that $\Omega$ is paracompact.  We
refer the reader to~\cite[IX.4]{Bourbaki:98} for the definition and
proprties of paracompact spaces; we only mention that metrizable and
compact spaces are paracompact.  Recall also that every open cover
$\{U_\alpha\}$ of a paracompact space $\Omega$ admits a
\term{partition of unity}, i.e., a collection $\{h_\alpha\}$ of
non-negative continuous functions such that
$\supp h_\alpha\subseteq U_\alpha$ for every~$\alpha$, for every point
$t\in\Omega$ there exists an open neighborhood $V$ of $t$ such that
all but finitely many functions in the collection are identically zero
on $V$, and $\sum_{\alpha}h_{\alpha}(t)=1$ for every $t\in\Omega$ (the
sum has only finitely many non-zero terms).

\begin{proposition}
  Suppose that $\Omega$ is paracompact Hausdorff space and
  $A\subseteq C(\Omega)_+$. Then $A$ is order bounded in $C(\Omega)$
  iff it is locally bounded, i.e., for every $t\in\Omega$ there exists
  an open neighborhood $U$ of $t$ and a real number $r\ge 0$ such that
  $f(s)\le r$ for all $f\in A$ and $s\in U$.
\end{proposition}

\begin{proof}
  The forward implication is obvious. To prove the converse, for every
  $n\in\mathbb N$, let $U_n$ be the union of all open sets $U$ such
  that $f(t)\le n$ for all $f\in A$ and $t\in U$. By the
  assumption, $\{U_n\}$ forms an open cover of $\Omega$. Let $\{h_n\}$
  be a corresponding partition of unity. For every $t$, put
  $g(t)=\sum_{n=1}^\infty nh_n(t)$. Note that for every $t$ there
  exists a neighborhood $V$ of $t$ such that only finitely many
  $h_n$'s are non-zero on $V$; it follows that $g$ is continuous. For
  every $t\in\Omega$ and $f\in A$, if $h_n(t)>0$ then $t\in U_n$ and,
  therefore, $f(t)\le n$. It follows that
  \begin{displaymath}
    g(t)=\sum_{n=1}^\infty nh_n(t)
    =\sum_{h_n(t)>0}nh_n(t)
    \ge\sum_{h_n(t)>0}f(t)h_n(t)=f(t).
  \end{displaymath}
\end{proof}

\begin{proposition}
  Suppose that $\Omega$ is a locally compact Hausdorff space and
  $A\subseteq C_0(\Omega)_+$. Then $A$ is order bounded iff $A$ is
  uniformly bounded and for every $\varepsilon>0$ there exists a
  compact set $K$ such that $f(t)<\varepsilon$ whenever $f\in A$ and
  $t\notin K$.
\end{proposition}
  
\begin{proof}
  The forward implication is trivial. To show the converse, suppose
  that $A$ satisfies the conditions in the theorem; we will show that
  $A$ is order bounded. Since $A$ is uniformly bounded, we may assume
  WLOG that $f(t)\le 1$ for all $f\in A$ and $t\in\Omega$. For every
  $n\in\mathbb N$, we find a compact set $K_n$ such that
  $f(t)<\frac{1}{2^n}$ whenever $f\in A$ and $t\notin K_n$.

  Recall that for every compact set $K$ in a locally compact space one
  can find a compact set $L$ such that $K\subseteq\Int L$. Applying
  this fact inductively, we find a sequence of compact sets $(L_n)$
  such that $K_n\subseteq L_n\subseteq\Int L_{n+1}$ for every $n$.
  Find a function $g_n\in C(\Omega)_+$ such that $g_n$
  equals 1 on $L_n$ and vanishes outside of $\Int L_{n+1}$. For every
  $t\in\Omega$, put $g(t)=\sum_{n=1}^\infty 2^{-n+1}g_n(t)$. Note that
  this sum is finite on each $L_n$ and vanishes outside
  $\bigcup_{n=1}^\infty L_n$; it is easy to see that $g\in
  C_0(\Omega)$.

  It is left to show that $A\le g$. Let $f\in A$. If $t\in L_1$ then
  $f(t)\le 1\le g_1(t)\le g(t)$. If $t\in L_{n}\setminus L_{n-1}$ for
  some $n>1$ then $t\notin K_{n-1}$, hence
  $f(t)<2^{-n+1}=2^{-n+1}g_n(t)\le g(t)$. Finally, if
  $t\notin\bigcup_{n=1}^\infty L_n$ then $f(t)=0=g(t)$.
\end{proof}

\bigskip

We now return to order convergence.  For the next few results, we will
assume that $\Omega$ is a completely regular Hausdorff Baire space. In
particular, $\Omega$ could be a locally compact Hausdorff space. The
equivalence
\eqref{inf-meagre-inf}$\Leftrightarrow$\eqref{inf-meagre-dense} in the
following lemma is a part of Corollary~4.14 in~\cite{Bilokopytov}. Cf
also Proposition~3.2 in~\cite{Kandic:19}.

\begin{lemma}\label{inf-meagre}
  Let $\Omega$ be a completely regular Hausdorff Baire space.
  For $G\subseteq C(\Omega)_+$, TFAE:
  \begin{enumerate}
  \item\label{inf-meagre-inf} $\inf G=0$;
  \item\label{inf-meagre-dense} There exists a dense set $D$ such that
    $\inf_{g\in G}g(t)=0$ for every $t\in D$;
  \item\label{inf-meagre-com} There exists a co-meagre set $D$ such that
    $\inf_{g\in G}g(t)=0$ for every $t\in D$;
  \end{enumerate}
\end{lemma}

\begin{proof}
  \eqref{inf-meagre-dense}$\Rightarrow$\eqref{inf-meagre-inf} by
  Lemma~\ref{CK-set-inf}.
  \eqref{inf-meagre-com}$\Rightarrow$\eqref{inf-meagre-dense} because
  every co-meagre set in a Baire space is dense.  To prove
  \eqref{inf-meagre-inf}$\Rightarrow$\eqref{inf-meagre-com}, assume
  that $\inf G=0$. For $n\in\mathbb N$, put
  $W_n=\bigcup_{g\in G}\{g<\frac1n\}$. Clearly, $W_n$ is open. For
  every non-empty open set $U$, Lemma~\ref{CK-set-inf} yields a point
  $t\in U$ and $g\in G$ with $g(t)<\frac1n$; hence $t\in W_n$. This
  means that $W_n$ is dense. Put $D:=\bigcap_{n=1}^\infty W_n$, then
  $D$ is co-meagre. Let $t\in D$, then for every
  $n\in\mathbb N$ we have $t\in W_n$, hence
  $\inf_{g\in G}g(t)\le\frac1n$. It follows that
  $\inf_{g\in G}g(t)=0$.
\end{proof}

The following result was proved in~\cite{vanderWalt:18} under extra
an assumption and for a different definition of order convergence.
The forward implication is a part of Corollary~4.14 in~\cite{Bilokopytov}. 

\begin{theorem}\label{CK-oconv-meag}
  Let $\Omega$ be a completely regular Hausdorff Baire space and
  $(f_\alpha)$ a net in $C(\Omega)$.  If $f_\alpha\goesuo f$ then
  $f_\alpha$ converges to $f$ pointwise on a co-meagre set. The
  converse is true for countable nets (in particular, for sequences).
\end{theorem}

\begin{proof}
  WLOG, $f=0$. Suppose that $f_\alpha\goesuo 0$. Then
  $\abs{f_\alpha}\wedge\one\goeso 0$. There exists a net
  $(g_\gamma)$ such that $g_\gamma\downarrow 0$ and for every $\gamma$
  there exists $\alpha_0$ such that $\abs{f_\alpha}\wedge\one\le g_\gamma$
  whenever $\alpha\ge\alpha_0$. Put $G=\{g_\gamma\}$. Then
  $\inf G=\inf g_\gamma=0$. By Lemma~\ref{inf-meagre}, there exists a
  co-meagre set $D$ such that for every $t\in D$ we have
  $0=\inf_{g\in G}g(t)=\inf_\gamma g_\gamma(t)$ and, therefore,
  $\lim_\alpha f_\alpha(t)=0$.

  For simplicity, we prove the converse implication for sequences;
  extending the proof to countable nets is straightforward.  Suppose
  that a sequence $(f_n)$ converges to zero on a co-meagre set~$D$. We
  will use Theorem~\ref{CK-o-conv} to show that $f_n\goesuo 0$. WLOG,
  $f_n\ge 0$ for all~$n$. Fix an open non-empty set $U$ and
  $\varepsilon>0$. For each $m$, put
  $W_m=\bigcup_{n\ge m}\{f_n>\varepsilon\}$. Clearly, $W_m$ is
  open. If $t\in\bigcap_mW_m$ then for every $m$ there exists $n\ge m$
  with $f_n(t)>\varepsilon$, hence $t\notin D$. This yields that
  $\bigcap_mW_m$ is contained in $\Omega\setminus D$, hence is
  meagre. Since $W_m$ is open, $\partial W_m$ is nowhere dense for
  every~$m$; we conclude that $\bigcup_m\partial W_m$ is meagre. It
  follows from
  \begin{displaymath}
    \bigcap_m\overline{W_m}\subset
    \Bigl(\bigcap_mW_m\Bigr)\cup\Bigl(\bigcup_m\partial W_m\Bigr)
  \end{displaymath}
  that $\bigcap_m\overline{W_m}$ is meagre, so that its complement
  \begin{math}
    \bigcup_m(\Omega\setminus\overline{W_m})
  \end{math}
  is co-meagre, hence dense, and, therefore, meets~$U$. It follows that
  $U\cap(\Omega\setminus\overline{W_m})\ne\varnothing$ for some~$m$. Denote
  this intersection by~$V$. For every $t\in V$, it follows from
  $t\notin W_m$ that $f_n(t)\le\varepsilon$ for all $n\ge m$. Hence,
  the condition in Theorem~\ref{CK-o-conv} is satisfied.
\end{proof}

\begin{example}
  The following example shows that one cannot replace ``co-meagre''
  with ``dense'' in Theorem~\ref{CK-oconv-meag}. Let $(f_n)$ be the
  Schauder system in $C[0,1]$; see,
  e.g.,~\cite[p.~3]{Lindenstrauss:77}.  Let $D$ be the set of all
  dyadic points in $[0,1]$. Then $D$ is dense in $[0,1]$ and
  $\bigl(f_n(t)\bigr)$ is eventually zero for every $t\in D$, so that
  $(f_n)$ converges to zero pointwise on~$D$. Yet, it is easy to see
  that $\sup_{n\ge m}f_n=\one$ in $C[0,1]$ for every~$m$, hence
  $(f_n)$ does not converge to zero in order.
\end{example}

\begin{example}
  We will construct an order bounded net that converges to zero
  pointwise at every point, yet fails to converge in order. This shows
  that the converse implication in Theorem~\ref{CK-oconv-meag}
  generally fails for nets. Let $\Omega=C[0,1]$; let $\Lambda$ be the
  set of all finite subsets of $[0,1]$ containing $0$ and $1$, ordered
  by inclusion. For each $\alpha\in\Lambda$,
  $\alpha=\{t_1,\dots,t_n\}$ with $0=t_0<t_1<\dots<t_n=1$, we define
  $f_\alpha$ as follows: we put $f_\alpha(t_i)=0$ for every $i$, we
  put $f_\alpha$ to be 1 at the midpoints, i.e.,
  $f_\alpha\bigl(\frac{t_{i-1}+t_{i}}{2}\bigr)=1$ for all
  $i=1,\dots,n$, and define $f_\alpha$ linearly in between. It is easy
  to see that $(f_\alpha)$ converges to zero pointwise on $[0,1]$. On
  the other hand, we claim that the supremum of every tail of this net
  in $C[0,1]$ is $\one$, hence the nets fails to converge to zero in
  order. Indeed, fix $\alpha\in\Lambda$, and suppose that
  $h\ge f_\beta$ for all $\beta\ge\alpha$. Let $s\notin\alpha$. Then
  one can find $\beta\supseteq\alpha$ such that $s$ is a midpoint for
  $\beta$. Hence, $h(s)\ge f_\beta(s)=1$. Thus, $h(s)\ge 1$ for all
  $s\notin\alpha$. Since $\alpha$ is a finite set and $h$ is
  continuous, it follows that $h\ge\one$.
\end{example}

\begin{corollary}\label{uo-Cauchy-meagre}
  Let $\Omega$ be a completely regular Hausdorff Baire space.  If a
  net $(f_\alpha)$ is uo-Cauchy then $\lim_\alpha f_\alpha(t)$ exists
  for every $t$ in some co-meagre set. The converse is true for
  countable nets.
\end{corollary}

\section{Convergence in $C^\infty(K)$}

Recall that a vector lattice $X$ is \term{universally complete} if it
is order complete and every disjoint set in $X$ has supremum. $X$ is
$\sigma$-\term{universally complete} if it is $\sigma$-order complete
and every disjoint sequence has a supremum. Let $K$ be a compact
Hausdorff space, then $C(K)$ is order complete iff $K$ is extremally
disconnected, i.e., the closure of every open set is open (hence
clopen). If $K$ is an extremally disconnected compact Hausdorff space
then the space $C^\infty(K)$ of all functions from $K$ to
$\overline{\mathbb R}=[-\infty,\infty]$ that are finite except on an
nowhere dense set is a universally complete vector lattice.  Moreover,
every universally complete vector lattice is lattice isomorphic to
$C^\infty(K)$ for some extremally disconnected compact Hausdorff
$K$. Every Archimedean vector lattice $X$ embeds as an order dense
sublattice into a unique universally complete vector lattice $X^u$,
called the \term{universal completion} of~$X$. We refer the reader
to~\cite{Aliprantis:03} for these and further details on $C^\infty(K)$ spaces
and universally complete vector lattices.

\begin{remark}\label{CKinfty}
  It can be easily verified that many of the preceding results,
  including Lemmas~\ref{CK-set-inf} and~\ref{inf-meagre},
  Theorems~\ref{CK-o-conv} and~\ref{CK-oconv-meag}, and
  Corollaries~\ref{CK-uo-Cauchy} and~\ref{uo-Cauchy-meagre} remain
  valid for $C^\infty(K)$. Note that the variant of
  Theorem~\ref{CK-oconv-meag} for $C^\infty(K)$ spaces was proved in
   \cite[XIII, Theorem 2.3.3]{Kantorovich:50}
\end{remark}

A set $A$ in an Archimedean vector lattice $X$ is said to be
\marg{new}
\term{dominable} if it is order bounded in $X^u$. Such sets appear
naturally in many applications. It is, therefore, important to
characterize order bounded sets in universally complete vector
lattices, or, equivalently, in $C^\infty(K)$-spaces. The following
result is somewhat motivated by Lemma~4 in~\cite{vanderWalt:12} and by
\cite[XIII. Theorem 2.3.2]{Kantorovich:50}

\begin{proposition}\label{o-bdd-in-C-infty}
  Let $A$ be a non-empty subset of $C^\infty(K)_+$ for some extremally
  disconnected $K$. $A$ is bounded above iff for every
  non-empty clopen subset $U$ of $K$ there exists a non-empty clopen
  $V\subseteq U$ and a non-negative real $r$ such that $f(t)\le r$ for all
  $f\in A$ and $t\in V$.
\end{proposition}

\begin{proof}
  The forward implication follows from the observation that every
  individual function in $C^\infty(K)_+$ satisfies the condition.


  Suppose now that $A$ satisfies the condition; we need to show that
  it is order bounded. As a partially ordered set, $\overline{\mathbb R}$ is
  order isomorphic to $[-1,1]$, which induces an order isomorphism
  between $C\bigl(K,\overline{\mathbb R}\bigr)$ and
  $C\bigl(K,[-1,1]\bigr)$. Hence $C(K,\overline{\mathbb R})$ is order
  complete as a partially ordered set. It follows that $h:=\sup A$ in
  $C(K,\overline{\mathbb R})$ exists. It suffices to show that
  $h\in C^\infty(K)$. Suppose not. Then there exists a non-empty
  clopen set $U$ such that $h$ equals infinity on~$U$. By assumption,
  we can find a non-empty clopen $V\subseteq U$ and $r\ge 0$ such that
  $f(t)\le r$ for all $f\in A$ and $t\in V$. Put
  \begin{displaymath}
    g(t)=
    \begin{cases}
      h(t) & \mbox{ when }t\notin V\\
      r & \mbox{ when }t\in V.
    \end{cases}
  \end{displaymath}
  Then $f\le g$ for all $f\in A$, hence $h\le g$, which contradicts
  $h$ being infinite on~$V$.
\end{proof}

We now present a simple proof of a theorem due to Grobler,
see~\cite{Gao:17}. We say that a net $(x_\alpha)_{\alpha\in\Lambda}$
has \term{finite heads} if the set
$\{\alpha\in\Lambda\mid\alpha\le\beta\}$ is finite for every
$\beta\in\Lambda$. Clearly, every sequence has finite heads.

\begin{proposition}\label{fin-heads}
  In a universally complete vector lattice, order and uo-convergences
  agree for nets with finite heads.
\end{proposition}

\begin{proof}
  Let $X$ be a universally complete vector lattice and $(x_\alpha)$ a
  net in $X$ with finite heads. We know that order convergence implies
  uo-convergence, so we only need to prove that if $x_\alpha\goesuo x$
  then $x_\alpha\goeso x$. WLOG, $x=0$ and $x_\alpha\ge 0$ for every
  $\alpha$. It suffices to show that $(x_\alpha)$ is order bounded.

  We may identify $X$ with a $C^\infty(K)$ space for some extremally
  disconnected compact Hausdorff~$K$. We will use
  Proposition~\ref{o-bdd-in-C-infty}.  Fix an open non-empty set
  $U\subseteq K$.  It follows from Theorem~\ref{CK-o-conv} and
  Remark~\ref{CKinfty} that there exists an open non-empty
  $V\subseteq U$ and $\alpha_0\in\mathbb N$ such that $x_\alpha$ is
  bounded by 1 on $V$ for all $\alpha\ge\alpha_0$. It now suffices to
  show that the head $(x_\alpha)_{\alpha\le\alpha_0}$ is uniformly
  bounded on an open subset of~$V$. Since this head has only finitely
  many terms, and since every function in $C^\infty(K)$ is continuous
  and finite except on a nowhere dense set, we can find an open
  non-empty $W\subseteq V$ and $r>0$ such that for every
  $\alpha\le\alpha_0$ the function $x_\alpha$ is bounded by $r$
  on~$W$. It follows from Proposition~\ref{o-bdd-in-C-infty} that
  $(x_\alpha)$ is order bounded.
\end{proof}

\begin{theorem}\label{Grobler}[Grobler]
  For a sequence $(x_n)$ in a $\sigma$-universally complete vector lattice, TFAE
  \begin{enumerate}
  \item\label{G-uo-con} $(x_n)$ is uo-convergent;
  \item\label{G-o-con} $(x_n)$ is order convergent;
  \item\label{G-uo-Cau} $(x_n)$ is uo-Cauchy;
  \item\label{G-o-Cau} $(x_n)$ is order Cauchy.
  \end{enumerate}
\end{theorem}

\begin{proof}
  Let $(x_n)$ be a sequence in a $\sigma$-universally complete vector
  lattice $X$. It is clear that
  \eqref{G-o-con}$\Rightarrow$\eqref{G-uo-con},
  \eqref{G-o-Cau}$\Rightarrow$\eqref{G-uo-Cau},
  \eqref{G-uo-con}$\Rightarrow$\eqref{G-uo-Cau}, and
  \eqref{G-o-con}$\Rightarrow$\eqref{G-o-Cau}. It suffices to show
  that \eqref{G-uo-Cau}$\Rightarrow$\eqref{G-o-con}. Suppose that
  $(x_n)$ is uo-Cauchy in $X$. Since $X$ is order dense and,
  therefore, regular in its universal completion~$X^u$, $(x_n)$ is
  uo-Cauchy in~$X^u$. By Proposition~\ref{fin-heads}, it is order
  Cauchy in~$X^u$. It follows that $(x_n)$ is order bounded in~$X^u$,
  hence is dominable in~$X$. Recall that by a theorem of Fremlin
  \cite[Theorem~7.38]{Aliprantis:03}, every countable dominable set in
  a $\sigma$-universally complete vector lattice is order
  bounded. This implies that $(x_n)$ is order bounded in~$X$. It
  follows that $(x_n)$ is order Cauchy in~$X$. Since $X$ is
  $\sigma$-order complete, $(x_n)$ is order convergent.
\end{proof}

\section{Convergence in vector lattices with PPP}

Next, we present an application of preceding results to vector
lattices with PPP. Recall that every Archimedean vector lattice $X$
with a strong unit $e$ may be represented as a norm dense sublattice
of $C(K)$ for some compact Hausdorff space~$K$, such that $e$
corresponds to~$\one$. It is easy to see that $X$ is also order dense
in $C(K)$.  Furthermore, the characteristic function of every clopen
set is contained in~$X$; see Proposition~4.1
in~\cite{Bilokopytov}. If, in addition, $X$ has PPP, then $K$ is
totally disconnected, i.e., every point has a base of clopen
neighborhoods; see Corollary~5.8 in~\cite{Bilokopytov}.

\begin{theorem}
  Let $X$ be a vector lattice with PPP and $(x_\alpha)$ a net
  in~$X_+$. Then $x_\alpha\goesuo 0$ iff for every $u>0$ and every
  real positive $\varepsilon$ there exists a non-zero component $v$ of
  $u$ and an index $\alpha_0$ such that $P_vx_\alpha\le\varepsilon u$
  whenever $\alpha\ge\alpha_0$.
\end{theorem}

\begin{proof}
  Suppose $x_\alpha\goesuo 0$. Let $u\in X_+$ and
  $\varepsilon\in(0,1)$.  Then $x_\alpha\wedge u\goesuo 0$. Since the
  principal ideal $I_u$ is regular in~$X$, we have
  $x_\alpha\wedge u\goesuo 0$ in~$I_u$. We can represent $I_u$ as a
  dense sublattice of $C(K)$ for some topological space $K$ such that
  $u$ becomes~$\one$. It is easy to see that $I_u$ is order dense in
  $C(K)$, hence $x_\alpha\wedge\one\goesuo 0$ in $C(K)$. By
  Theorem~\ref{CK-o-conv}, there exists a non-empty open
  $W\subseteq K$ and an index $\alpha_0$ such that
  $x_\alpha\le\varepsilon$ on $W$ for all $\alpha\ge\alpha_0$.

  Since $X$ has PPP, so does~$I_u$. It follows that $K$ is totally
  disconnected and, therefore, there exists a non-empty clopen set $V$
  contained in~$W$. Put $v=\one_V$, then $v\in C(K)$, $v$ is a
  component of $u$ and $v\in I_u$, hence we may view $v$ as an element
  of~$X$. It is easy to see that for every $\alpha\ge\alpha_0$ we have
  $P_vx_\alpha\le\varepsilon\one$ in $C(K)$, hence
  $P_vx_\alpha\le\varepsilon u$ in $I_u$ and, therefore, in~$X$.

  To prove the converse, consider the universal completion $X^u$
  of~$X$. Then $X^u$ may be represented as $C^\infty(K)$ for some
  extremally disconnected Hausdorff compact~$K$. It suffices to show
  that $x_\alpha\goesuo 0$ in $C^\infty(K)$. We will do this by
  proving that $(x_\alpha)$ satisfies the condition in
  Theorem~\ref{CK-o-conv} (see also Remark~\ref{CKinfty}).

  Fix a non-empty open subset $U$ of $K$ and $\varepsilon>0$. Since
  $X$ is order dense in~$X^u$, there exists $0<u\in X$ such that
  $u\le\one_U$. By assumption, we can find a non-zero component $v$ of
  $u$ such that $P_vx_\alpha\le\varepsilon u$. Find a non-empty clopen
  set $V$ and $\delta>0$ such that $\delta\one_V\le v$ whenever
  $\alpha\ge\alpha_0$. Then $P_Vx_\alpha\le P_vx_\alpha\le\varepsilon
  u\le\varepsilon\one_U$. It follows that
  ${x_\alpha}_{|V}\le\varepsilon$.
  %
  %
\end{proof}

\section{Convergence in Boolean algebras}

The concept of order converges extends naturally to partially ordered
sets as follows: $x_\alpha\goeso x$ if there exist non-empty sets $A$
and $B$ such that $\sup A=x=\inf B$ and for every $a\in A$ and
$b\in B$, the order interval $[a,b]$ contains a tail of the net.
Lemma~\ref{oconv-dom-set} ensures that in vector lattices this
definition agrees with the old one. The following lemma is
straightforward:

\begin{lemma}\label{event-bds}
  For a net $(x_\alpha)$ in a partially ordered set, put 
  \begin{displaymath}
    A_0=\bigl\{a\mid\exists\alpha_0\ \forall\alpha\ge\alpha_0\
    a\le x_\alpha\bigr\}\quad\mbox{and}\quad
    B_0=\bigl\{b\mid\exists\alpha_0\ \forall\alpha\ge\alpha_0\
         b\ge x_\alpha\bigr\}.
  \end{displaymath}
  Then $x_\alpha\goeso x$ iff $A_0$ and $B_0$ are non-empty and
  $\sup A_0=x=\inf B_0$.
\end{lemma}

The reader may have noticed that many of the proofs in the preceding
sections implicitly relied on Boolean algebras associated with the
topological space. To make this connection more explicit, we
will now extend Lemma~\ref{CK-set-inf}, Theorem~\ref{CK-o-conv} and
Corollary~\ref{CK-uo-Cauchy} to Boolean algebras using essentially the
same techniques as before.

We refer the reader
to~\cite{Vladimirov:02} for background on Boolean algebras. The symbol
$\mathcal A$ will stand for a Boolean algebra. We write $\bar a$ for
the complement of $a$, $a-b$ for the difference $a\wedge\bar b$ and
$\abs{a-b}$ for the symmetric difference
$(a\wedge\bar b)\vee(b\wedge\bar a)$. Note that $x\wedge a=y\wedge a$
iff $\abs{x-y}\perp a$ iff $\abs{x-y}\le\bar a$. Also, $\abs{x-a}\le
z$ iff $a-z\le x\le a\vee x$. 

Since every net in a Boolean algebra is order bounded, uo-convergence
and order convergence agree. It follows easily from the definition of
order convergence that $x_\alpha\goeso x$ iff
$\bar x_\alpha\goeso\bar x$. Since lattice operations in a Boolean
algebra satisfy infinite distributive laws, it is easy to see that
$x_\alpha\goeso x$ implies $x_\alpha\wedge c\goeso x\wedge c$ and
$x_\alpha\vee c\goeso x\vee c$ for every $c$.

\begin{lemma}\label{sym-diff}
  $x_\alpha\goeso x$ iff $\abs{x_\alpha-x}\goeso 0$.
\end{lemma}

\begin{proof}
  Suppose that $x_\alpha\goeso x$ and let $A$ and $B$ be as in the
  definition of order convergence. Put
  $C=\bigl\{b-a\mid a\in A, b\in B\bigr\}$.  It is easy to see that
  $\inf C=0$. For $a\in A$ and $b\in B$ there exists $\alpha_0$ such
  that for every $\alpha\ge\alpha_0$, $a\le x_\alpha\le b$. It follows
  that $\abs{x-x_\alpha}\le b-a$ and, therefore,
  $\abs{x_\alpha-x}\goeso 0$.

  Suppose that $\abs{x_\alpha-x}\goeso 0$.  We can then find a set $C$
  with $\inf C=0$ and for every $c\in C$ there exists $\alpha_0$ such
  that for all $\alpha\ge\alpha_0$ we have $\abs{x_\alpha-x}\le c$ or,
  equivalently, $x-c\le x_\alpha\le x\vee c$. Put
  $A=\{x-c\mid c\in C\}$ and $B=\{x\vee c\mid c\in C\}$; then
  $\sup A=x=\inf B$. It follows that $x_\alpha\goeso x$.
\end{proof}

\begin{lemma}\label{BA-set-inf}
  Let $B\subseteq\mathcal A$. Then $\inf B=0$ iff for every $u>0$ there exists
  a non-zero $v\le u$ and $b\in B$ such that $v\perp b$. Also, $\sup B=1$ iff
  for every $u>0$ there exists $0<v\le u$ and $b\in B$ such that
  $v\le b$.
\end{lemma}

\begin{proof}
  We only prove the first claim; the second is similar. Suppose that
  $\inf B=0$ and $u>0$. Then there exists $b\in B$ such that $u\not\le
  b$. It follows that $v\perp b$ where $v=u-b>0$. Conversely, suppose
  that $\inf B=0$ fails. Then there exists $u$ with $0<u\le B$. Then
  for every $0<v\le u$ and $b\in B$ we have $v\le b$, so that
  $v\not\perp b$.
\end{proof}

\begin{theorem}\label{BA-o-conv}
  For a net $(x_\alpha)$ in a Boolean algebra $\mathcal A$,
  $x_\alpha\goeso x$ iff for every $u>0$ there exists $0<v\le u$ and
  an index $\alpha_0$ such that $x_\alpha\wedge v=x\wedge v$ for all
  $\alpha\ge\alpha_0$.
\end{theorem}

\begin{proof}
  Suppose that $x_\alpha\goeso x$ and $u>0$. Then
  $\abs{x_\alpha-x}\goeso 0$. Take $C\subseteq A$ with
  $\inf C=0$ which ``witnesses'' the latter convergence. By
  Lemma~\ref{BA-set-inf}, there exists $v\in\mathcal A$ and $c\in C$
  such that $0<v\le u$ and $v\perp c$. There exists $\alpha_0$ such
  that for every $\alpha\ge\alpha_0$ we have $\abs{x_\alpha-x}\le
  c\le\bar v$, and, therefore, $x_\alpha\wedge v=x\wedge v$.



  To prove the converse, for every $u>0$ we fix $v_u$ 
   and $\alpha_u$ such that $0<v_u\le u$ for all $\alpha\ge\alpha_u$
   we have $ x_\alpha\wedge
   v_u=x\wedge v_u$, or, equivalently, $\abs{x-x_\alpha}\le
   \bar v_u$. Let $C=\{\bar v_u\mid u>0\}$. By Lemma~\ref{BA-set-inf}, we
   have $\inf C=0$. We conclude that $\abs{x_\alpha-x}\goeso 0$ and,
   therefore, $x_\alpha\goeso x$.
\end{proof}

\begin{corollary}
  For a net $(x_\alpha)$ in a Boolean algebra, TFAE:
  \begin{enumerate}
  \item\label{BA-C-C} $(x_\alpha)$ is order Cauchy;
  \item\label{BA-C-stab} For every $u>0$ there exists $0<v\le u$ and
    $\alpha_0$ such that $x_\alpha\wedge v=x_\beta\wedge v$ for all
    $\alpha,\beta\ge\alpha_0$;
  \item\label{BA-C-01}
    \begin{math}
      \forall u>0\
      \Bigl(\exists v\in(0,u]\mbox{ and }\alpha_0\
       \forall\alpha\ge\alpha_0\ v\le x_\alpha\Bigr)\mbox{ or }\\
      \Bigl(\exists v\in(0,u]\mbox{ and }\alpha_0\
       \forall\alpha\ge\alpha_0\ v\le\bar x_\alpha\Bigr)
    \end{math}
   \item\label{BA-C-nots} $\sup(A_0\cup\overline{B_0})=1$, where $A_0$ and
    $B_0$ are as in Lemma~\ref{event-bds};
   \item\label{BA-C-AB} There exists non-empty sets $A$ and $B$ such
     that $\sup(A\cup\overline{B})=1$ and for every $a\in A$ and $b\in B$
     there exists $\alpha_0$ with $a\le x_\alpha\le b$ for all
     $\alpha\ge\alpha_0$. 
  \end{enumerate}
\end{corollary}

\begin{proof}
  \eqref{BA-C-C}$\Leftrightarrow$\eqref{BA-C-stab}
  $(x_\alpha)$ is order Cauchy iff $\abs{x_\alpha-x_\beta}\to 0$. By
  Theorem~\ref{BA-o-conv}, the latter is equivalent to
  $\forall u>0$ $\exists v\in(0,u]$ and $\alpha_0$
  $\forall\alpha,\beta\ge\alpha_0$ $\abs{x_\alpha-x_\beta}\perp
  v$ or, equivalently, $x_\alpha\wedge v=x_\beta\wedge v$. 

  \eqref{BA-C-01}$\Rightarrow$\eqref{BA-C-stab} trivially. To show
  \eqref{BA-C-stab}$\Rightarrow$\eqref{BA-C-01}, let $u$, $v$, and
  $\alpha_0$ be in \eqref{BA-C-stab}. Put $w=x_{\alpha_0}\wedge v$; then
  $(x_\alpha\wedge v)=w$ for all $\alpha\ge\alpha_0$. If $w=0$, we are
  done. Otherwise, replace $v$ with~$w$. 
  
  \eqref{BA-C-01}$\Leftrightarrow$\eqref{BA-C-nots}:
  \eqref{BA-C-01} effectively says that for every $u>0$ there exists
  $v\in(0,u]$ such that $v\in A_0\cup\overline{B_0}$. This is
  equivalent to  $\sup(A_0\cup\overline{B_0})=1$ by Lemma~\ref{BA-set-inf}.

  \eqref{BA-C-nots}$\Leftrightarrow$\eqref{BA-C-AB}
  is straightforward. 
\end{proof}

\end{document}